\documentclass{amsart}
\usepackage{amsmath}
\usepackage{amssymb}
\usepackage{amsfonts}
\usepackage{graphicx}

\makeatletter \@namedef{subjclassname@2010}{  \textup{2010}
Mathematics Subject Classification}
\theoremstyle{plain}
\newtheorem{theorem}{Theorem}[section]
\newtheorem{corollary}[theorem]{Corollary}
\newtheorem{proposition}[theorem]{Proposition}
\newtheorem{lemma}[theorem]{Lemma}
\theoremstyle{remark}
\newtheorem{remark}[theorem]{Remark}
\theoremstyle{definition}

\theoremstyle{example}

\numberwithin{equation}{section}
\input{tcilatex}

\begin{document}
\title[WEIGHT ERGODIC THEOREMS FOR PROBABILITY MEASURES]{WEIGHT ERGODIC
THEOREMS FOR PROBABILITY MEASURES ON LOCALLY COMPACT GROUPS }
\author{H. S. MUSTAFAYEV}
\address{Van Yuzuncu Yil University, Faculty of Science, Department of
Mathematics, VAN-TURKEY}
\email{hsmustafayev@yahoo.com}

\begin{abstract}
Let $G$ be a locally compact group with the left Haar measure $m_{G}$. A
probability measure $\mu $ on $G$ is said to be \textit{strictly aperiodic}
if the support of $\mu $ is not contained in a proper closed left coset of $%
G.$ Assume that $\mu $ is a strictly aperiodic measure on $G$ and $\left\{
a_{n}\right\} _{n\in
\mathbb{N}
}$ is a bounded good weight for the mean ergodic theorem \cite[Definition
21.1]{8}. We show that if $G$ is compact, then

\begin{equation*}
\text{w}^{\ast }-\lim_{n\rightarrow \infty }\frac{1}{n}\sum_{i=1}^{n}a_{i}%
\mu ^{i}=\left( \lim_{n\rightarrow \infty }\frac{1}{n}\sum_{i=1}^{n}a_{i}%
\right) \overline{m}_{H},
\end{equation*}%
where $H$ is the closed subgroup of $G$ generated by $\func{supp}\mu $ and $%
\overline{m}_{H}$ is the measure on $G$ defined by $\overline{m}_{H}\left(
E\right) =m_{H}\left( E\cap H\right) $ for every Borel subset $E$ of $G.$ If
$H$ is not compact, then%
\begin{equation*}
w^{\ast }-\lim_{n\rightarrow \infty }\frac{1}{n}\sum_{i=1}^{n}a_{i}\mu
^{i}=0.
\end{equation*}%
Some related problems are also discussed.
\end{abstract}

\maketitle

\section{Introduction}

Let $G$ be a locally compact group with the left Haar measure $m_{G}$ (in
the case when $G$ is compact, $m_{G}$ will denote normalized Haar measure on
$G$) and let $M\left( G\right) $ be the convolution measure algebra of $G.$
As usual, $C_{0}\left( G\right) $ will denote the space of all complex
valued continuous functions on $G$ vanishing at infinity. Since $C_{0}\left(
G\right) ^{\ast }=M\left( G\right) ,$ the space $M\left( G\right) $ carries
the weak$^{\ast }$ topology $\sigma \left( M\left( G\right) ,C_{0}\left(
G\right) \right) .$ In the following, the w$^{\ast }$-topology on $M\left(
G\right) $ always means this topology. For a subset $S$ of $G,$ by $\left[ S%
\right] $ we will denote the closed subgroup of $G$ generated by $S.$ A
probability measure $\mu $ on $G$ is said to be \textit{adapted} if $\left[
\func{supp}\mu \right] =G.$ Also, a probability measure $\mu $ on $G$ is
said to be \textit{strictly aperiodic} if the support of $\mu $ is not
contained in a proper closed left coset of $G.$ For $n\in
\mathbb{N}
,$ by $\mu ^{n}$ we will denote $n$-th convolution power of $\mu \in M\left(
G\right) $, where $\mu ^{0}:=\delta _{e}$ is the Dirac measure concentrated
at the unit element $e$ of $G.$

A classical Kawada-It\^{o} theorem \cite[Theorem 7]{14} asserts that if $\mu
$ is an adapted measure on a compact metrizable group $G,$ then the sequence
of probability measures $\left\{ \frac{1}{n}\sum_{i=1}^{n}\mu ^{i}\right\}
_{n\in
\mathbb{N}
}$ weak$^{\ast }$ converges to the Haar measure on $G$ (see also \cite[%
Theorem 3.2.4]{11}). If $\mu $ is an adapted and strictly aperiodic measure
on a compact metrizable group $G,$ then w$^{\ast }$-$\lim_{n\rightarrow
\infty }\mu ^{n}=m_{G}$ \cite[Theorem 8]{14}. If $\mu $ is an adapted
measure on a second countable non-compact group $G$, then w$^{\ast }$-$%
\lim_{n\rightarrow \infty }\mu ^{n}=0$ \cite[Theorem 2]{18}. In \cite[Th\'{e}%
or\`{e}me 8]{3}, it was proved that if $\mu $ is a strictly aperiodic
measure on a non-compact locally compact group $G,$ then w$^{\ast }$-$%
\lim_{n\rightarrow \infty }\mu ^{n}=0$. For related results see also, \cite%
{1}, \cite{2}, \cite{11}.

In this paper, we present some results concerning convergence of probability
measures on locally compact groups.

\section{Weak$^{\ast }$ convergence}

Let $X$ be a complex Banach space and let $B\left( X\right) $ be the algebra
of all bounded linear operators on $X$. An operator $T\in B\left( X\right) $
is said to be \textit{mean ergodic} if the limit
\begin{equation*}
\lim_{n\rightarrow \infty }\frac{1}{n}\sum_{i=1}^{n}T^{i}x\text{ \ exists in
norm }\forall x\in X.
\end{equation*}%
If $T$ is mean ergodic, then
\begin{equation*}
P_{T}x:=\lim_{n\rightarrow \infty }\frac{1}{n}\sum_{i=1}^{n}T^{i}x\text{ \ }%
\left( x\in X\right)
\end{equation*}%
is the projection onto $\ker \left( T-I\right) $. The projection $P_{T}$
will be called \textit{mean ergodic projection associated with\ }$T.$

An operator $T\in B\left( X\right) $ is said to be \textit{power bounded} if
\begin{equation*}
C_{T}:=\sup_{n\geq 0}\left\Vert T^{n}\right\Vert <\infty .
\end{equation*}%
A power bounded operator $T$ on a Banach space $X$ is mean ergodic if and
only if%
\begin{equation}
X=\overline{\text{ran}\left( T-I\right) }\oplus \ker \left( T-I\right) .
\label{2.1}
\end{equation}%
It is well known that a power bounded operator on a reflexive Banach space
is mean ergodic (for instance see, \cite[Chapter 2]{15}).

The following result is an immediate consequence of the identity (2.1).

\begin{proposition}
Let $T$ be a power bounded operator on a Banach space $X$ and assume that $%
\lim_{n\rightarrow \infty }\left\Vert T^{n+1}x-T^{n}x\right\Vert =0$ for all
$x\in X.$ If $T$ is mean ergodic $($so if $X$ is reflexive$)$, then $%
\lim_{n\rightarrow \infty }\left\Vert T^{n}x-P_{T}x\right\Vert =0$ for all $%
x\in X,$ where $P_{T}$ is the \textit{mean ergodic projection }associated
with\textit{\ }$T$.
\end{proposition}

The open unit disc and the unit circle in the complex plane will be denoted
by $\mathbb{D}$ and $\mathbb{T}$, respectively. If $T\in B\left( X\right) $
is power bounded then clearly, $\sigma \left( T\right) \subseteq \overline{%
\mathbb{D}}$, where $\sigma \left( T\right) $ is the spectrum of $T.$ The
classical Katznelson-Tzafriri theorem \cite{13} states that if $T\in B\left(
X\right) $ is power bounded, then $\lim_{n\rightarrow \infty }\left\Vert
T^{n+1}-T^{n}\right\Vert =0$ if and only if $\sigma \left( T\right) \cap
\mathbb{T}\subseteq \left\{ 1\right\} .$

If $T\in B\left( X\right) $ is mean ergodic, then $T$ is Ces\`{a}ro bounded,
that is,
\begin{equation*}
\sup_{n\in
\mathbb{N}
}\left\Vert \frac{1}{n}\sum_{i=1}^{n}T^{i}\right\Vert <\infty .
\end{equation*}%
It follows from the spectral mapping theorem that $r\left( T\right) \leq 1,$
where $r\left( T\right) $ is the spectral radius of $T.$

Let $N$ be a normal operator on a complex Hilbert space $\mathcal{H}$ with
the spectral measure $E$ $($recall that $\sigma \left( N\right) =\func{supp}%
E $). If $N$ is mean ergodic, then $\left\Vert N\right\Vert =r\left(
N\right) \leq 1$ ( notice also that a normal operator is power bounded if
and only\ it is a contraction). It easily follows from the spectral theorem
that if $N$ is a normal contraction operator on a Hilbert space $\mathcal{H}$%
, then%
\begin{equation*}
\lim_{n\rightarrow \infty }\frac{1}{n}\sum_{i=i}^{n}N^{i}x=E\left\{
1\right\} x\text{\ \ in norm }\forall x\in \mathcal{H}.
\end{equation*}%
On the other hand, a few lines of computation show that
\begin{equation*}
\lim_{n\rightarrow \infty }\left\Vert N^{n+1}x-N^{n}x\right\Vert
^{2}=\dint\limits_{\sigma \left( N\right) \cap \mathbb{T\diagdown }\left\{
1\right\} }\left\vert z-1\right\vert ^{2}d\langle E\left( z\right)
x,x\rangle \text{, \ }\forall x\in \mathcal{H}\text{.}
\end{equation*}%
From this identity and from Proposition 2.1 it follows that the sequence $%
\left\{ N^{n}x\right\} _{n\in
\mathbb{N}
}$ converges in norm for every $x\in \mathcal{H}$ if and only if $\left\Vert
N\right\Vert \leq 1$ and $E\left[ \sigma \left( N\right) \cap \mathbb{%
T\diagdown }\left\{ 1\right\} \right] =0.$ Under these conditions we have
\begin{equation*}
\lim_{n\rightarrow \infty }\left\Vert N^{n}x-E\left\{ 1\right\} x\right\Vert
=0\text{, \ }\forall x\in \mathcal{H}.
\end{equation*}

Recall from \cite[Definition 21.1]{8} that a sequence $\left\{ a_{n}\right\}
_{n\in
\mathbb{N}
}$ in $%
\mathbb{C}
$ is called \textit{good weight} \textit{for the mean ergodic theorem}
(briefly \textit{good weight}) if for every Hilbert space $\mathcal{H}$ and
every contraction $T$ on $\mathcal{H}$ the limit
\begin{equation*}
\lim_{n\rightarrow \infty }\frac{1}{n}\sum_{i=1}^{n}a_{i}T^{i}x\text{ \
exists }\forall x\in \mathcal{H}\text{.}
\end{equation*}%
By Theorem 21.2 of \cite{8}, a bounded sequence $\left\{ a_{n}\right\}
_{n\in
\mathbb{N}
}$ is a good weight if and only if the limit
\begin{equation*}
\lim_{n\rightarrow \infty }\frac{1}{n}\sum_{i=1}^{n}a_{i}\xi ^{i}=:a\left(
\xi \right) \text{ \ exists }\forall \xi \in \mathbb{T}\text{.}
\end{equation*}

Let $\left( \Omega ,\Sigma ,m\right) $ be a probability space and let $%
\varphi :\Omega \rightarrow \Omega $ be a measure-preserving transformation.
It follows from the Wiener-Wintner theorem \cite[Corollary 21.6]{8} that the
sequence $\left( f\left( \varphi ^{n}\left( \omega \right) \right) \right)
_{n\in
\mathbb{N}
}$ is a bounded good weight for any $f\in L^{\infty }\left( \Omega \right) $
and for all almost every $\omega \in \Omega $.

If $\left\{ a_{n}\right\} _{n\in
\mathbb{N}
}$ is a bounded good weight, then for every contraction $T$ on a Hilbert
space $\mathcal{H}$ one has
\begin{equation}
\lim_{n\rightarrow \infty }\frac{1}{n}\sum_{i=1}^{n}a_{i}T^{i}x=\sum_{\xi
\in \sigma _{p}\left( T\right) \cap \mathbb{T}}a\left( \xi \right) P_{\xi }x%
\text{\ \ in norm }\forall x\in \mathcal{H},  \label{2.2}
\end{equation}%
where $P_{\xi }$ are the orthogonal projections onto the mutually orthogonal
eigenspaces $\ker \left( T-\xi I\right) $ for $\xi \in \sigma _{p}\left(
T\right) \cap \mathbb{T}$, where $\sigma _{p}\left( T\right) $ is the point
spectrum of $T$ \cite[Theorem 21.2]{8} (it follows that $a\left( \xi \right)
\neq 0$ for at most countably many $\xi \in \mathbb{T}$).

Let $N$ be a normal contraction operator on a Hilbert space $\mathcal{H}$
with the spectral measure $E.$ Assume that $E\left\{ \xi \right\} =0$ for
all $\xi \in \mathbb{T}\diagdown \left\{ 1\right\} .$ It follows from the
identity (2.2) that if $\left\{ a_{n}\right\} _{n\in
\mathbb{N}
}$ is a bounded good weight, then%
\begin{equation*}
\lim_{n\rightarrow \infty }\frac{1}{n}\sum_{i=1}^{n}a_{i}N^{i}x=\left(
\lim_{n\rightarrow \infty }\frac{1}{n}\sum_{i=1}^{n}a_{i}\right) E\left\{
1\right\} x\text{\ \ in norm }\forall x\in \mathcal{H}\text{.}
\end{equation*}

Let $G$ be a locally compact group with the left Haar measure $m_{G}$. If $H$
is a closed subgroup of $G$, then the measure $m_{H}$ may be regarded as a
measure on $G$ by putting $\overline{m}_{H}\left( E\right) =m_{H}\left(
E\cap H\right) $ for every Borel subset $E$ of $G.$ Note that $\func{supp}%
\overline{m}_{H}=H.$

One of the main results of the paper is the following.

\begin{theorem}
Let $\mu $ be a strictly aperiodic measure on a locally compact group $G$
and let $\left\{ a_{n}\right\} _{n\in
\mathbb{N}
}$ be a bounded good weight. The following assertions hold:

$\left( a\right) $ If $G$ is compact, then
\begin{equation*}
w^{\ast }-\lim_{n\rightarrow \infty }\frac{1}{n}\sum_{i=1}^{n}a_{i}\mu
^{i}=\left( \lim_{n\rightarrow \infty }\frac{1}{n}\sum_{i=1}^{n}a_{i}\right)
\overline{m}_{\left[ \func{supp}\mu \right] }.
\end{equation*}

$\left( b\right) $ If $\left[ \func{supp}\mu \right] $ is not compact, then
\begin{equation*}
w^{\ast }-\lim_{n\rightarrow \infty }\frac{1}{n}\sum_{i=1}^{n}a_{i}\mu
^{i}=0.
\end{equation*}
\end{theorem}

Below, we collect some preliminary results that will needed later on.

Let $G$ be a locally compact group. A \textit{representation} $\pi $ of $G$
on a Banach space $X_{\pi }$ (the representation space of $\pi )$ is a
homomorphism from $G$ to the group of invertible isometries on $X_{\pi }.$
We assume that $\pi $ is strongly continuous. Then, for an arbitrary $\mu
\in M\left( G\right) $ we can define a bounded linear operator $\widehat{\mu
}\left( \pi \right) $ on $X_{\pi }$ by
\begin{equation*}
\widehat{\mu }\left( \pi \right) x=\int_{G}\pi \left( g\right) xd\mu \left(
g\right) \text{, \ }x\in X_{\pi }.
\end{equation*}%
The map $\mu \rightarrow \widehat{\mu }\left( \pi \right) $ is linear,
multiplicative, and contractive; $\left\Vert \widehat{\mu }\left( \pi
\right) \right\Vert \leq \left\Vert \mu \right\Vert _{1},$ where $\left\Vert
\mu \right\Vert _{1}$ is the total variation norm of $\mu \in M\left(
G\right) .$

By $\widehat{G}$ we will denote unitary dual of $G$, the set of all
equivalence classes of irreducible continuous unitary representations of $G$
with the Fell topology. Recall that $\pi _{0}\in \widehat{G}$ is a limit
point of $M\subset \widehat{G}$ in the Fell topology, if the matrix function
$g\rightarrow \langle \pi _{0}\left( g\right) x_{0},x_{0}\rangle $ $\left(
x_{0}\in \mathcal{H}_{\pi _{0}}\right) $ can be uniformly approximated on
every compact subset of $G$ by the matrix functions $g\rightarrow \langle
\pi \left( g\right) x,x\rangle $ $\left( \pi \in M,\text{ }x\in \mathcal{H}%
_{\pi }\right) $ (in the case when $G$ is abelian, Fell topology coincides
with the usual topology of $\widehat{G},$ the dual group of $G$). The
function $\pi \rightarrow \widehat{\mu }\left( \pi \right) $ $\left( \pi \in
\widehat{G}\right) $ is called \textit{Fourier-Stieltjes transform }of $\mu
\in M\left( G\right) $ (if $G$ is abelian, the Fourier-Stieltjes transform
of $\mu \in M\left( G\right) $ will be denoted simply by $\widehat{\mu }$).
If $\widehat{\mu }\left( \pi \right) =0$ for all $\pi \in \widehat{G},$ then
$\mu =0$ (for instance see, \cite[\S 18]{5}). It is well known that if $G$
is compact, then every $\pi \in \widehat{G}$ is finite dimensional. Also, we
know that if $G$ is compact (resp. compact and metrizable) then $\widehat{G}$
is discrete (resp. countable). These facts are consequences of the
Peter-Weyl theory \cite[Chapter 4]{17}.

By $B_{X}$ and $S_{X}$ respectively, we denote the closed unit ball and the
unit sphere of a Banach space $X.$ Notice that ext$B_{X}\subseteq S_{X},$
where ext$B_{X}$ is the set of all extreme points of $B_{X}.$ For the sake
of convenience, $X$ will be called $A$-\textit{space} if ext$B_{X}=S_{X}$.
For example, uniformly convex Banach spaces, in particular, Hilbert spaces
and $L^{p}$ $\left( 1<p<\infty \right) $ spaces are $A$-spaces.

The proof of the following result is similar to the proof of \cite[%
Proposition 2.1]{4}.

\begin{lemma}
Let $\mu $ be a probability measure on a locally compact group $G$ and let $%
\pi $ be a continuous Banach representation of $G$. If the representation
space $X_{\pi }$ is a Banach $A$-space, then for an arbitrary $\xi \in
\mathbb{T}$ we have
\begin{equation*}
\ker \left[ \widehat{\mu }\left( \pi \right) -\xi I_{\pi }\right] =\left\{
x\in X_{\pi }:\pi \left( g\right) x=\xi x,\text{ }\forall g\in \func{supp}%
\mu \right\} .
\end{equation*}
\end{lemma}

\begin{proof}
Assume that $\widehat{\mu }\left( \pi \right) x=\xi x$ for some $x\in X_{\pi
}$ with $\left\Vert x\right\Vert =1$ and $\xi \in \mathbb{T}.$ Let $s\in $$%
\func{supp}\mu $ and let $\left\{ U_{i}\right\} _{i\in I}$ be a directed (by
reverse inclusion) basis neighborhood system of $s.$ Then $\mu \left(
U_{i}\right) >0$ for all $i\in I.$ There is no restriction to assume that $%
\mu \left( U_{i}\right) <1$ for all $i\in I.$ If
\begin{equation*}
x_{i}:=\frac{1}{\mu \left( U_{i}\right) }\int_{U_{i}}\pi \left( g\right)
xd\mu \left( g\right) \text{ \ and \ \ }y_{i}:=\frac{1}{1-\mu \left(
U_{i}\right) }\int_{G\diagdown U_{i}}\pi \left( g\right) xd\mu \left(
g\right) ,
\end{equation*}%
then $\left\Vert x_{i}\right\Vert \leq 1$ and $\left\Vert y_{i}\right\Vert
\leq 1$. The identity
\begin{equation*}
\xi x=\mu \left( U_{i}\right) x_{i}+\left[ 1-\mu \left( U_{i}\right) \right]
y_{i}
\end{equation*}%
shows that $\xi x$ is a convex combination of $x_{i}$ and $y_{i}.$ Since $%
\left\Vert \xi x\right\Vert =1,$ we have $\xi x=x_{i}$ for all $i\in I$ and
therefore,%
\begin{equation*}
\frac{1}{\mu \left( U_{i}\right) }\int_{U_{i}}\left[ \pi \left( g\right)
x-\xi x\right] d\mu \left( g\right) =0\text{, \ }\forall i\in I.
\end{equation*}%
As $U_{i}\rightarrow \left\{ s\right\} $ we obtain $\pi \left( s\right)
x=\xi x.$
\end{proof}

The following result was proved in \cite[Lemma 2.3]{20}.

\begin{lemma}
Let $\mu $ be a strictly aperiodic measure on a locally compact group $G$
and let $\pi $ be a continuous unitary representation of $G$. Then, the
operator $\widehat{\mu }\left( \pi \right) $ cannot have unitary eigenvalues
except $\xi =1.$
\end{lemma}

Now, let $\mu $ be a strictly aperiodic measure on a locally compact group $%
G $ and let $\pi $ be a continuous unitary representation of $G$. By Lemma
2.4, the operator $\widehat{\mu }\left( \pi \right) $ cannot have unitary
eigenvalues except $\xi =1.$ In view of the identity (2.2) and Lemma 2.3, we
have the following.

\begin{corollary}
Let $\mu $ be a strictly aperiodic measure on a locally compact group $G$
and let $\pi $ be a continuous unitary representation of $G$ on a Hilbert
space $\mathcal{H}_{\pi }$. If $\left\{ a_{n}\right\} _{n\in
\mathbb{N}
}$ is a bounded good weight, then
\begin{equation*}
\lim_{n\rightarrow \infty }\frac{1}{n}\sum_{i=1}^{n}a_{i}\widehat{\mu }%
\left( \pi \right) ^{i}x=\left( \lim_{n\rightarrow \infty }\frac{1}{n}%
\sum_{i=1}^{n}a_{i}\right) P_{\mu }^{\pi }x\text{ \ in norm }\forall x\in
\mathcal{H}_{\pi }\text{,}
\end{equation*}%
where $P_{\mu }^{\pi }$ is the orthogonal projection onto
\begin{equation*}
\left\{ x\in \mathcal{H}_{\pi }:\pi \left( g\right) x=x:\forall g\in \left[
\func{supp}\mu \right] \right\} .
\end{equation*}%
In addition if $\mu $ is adapted and $\pi \in \widehat{G}\diagdown id,$ then
\begin{equation*}
\lim_{n\rightarrow \infty }\frac{1}{n}\sum_{i=1}^{n}a_{i}\widehat{\mu }%
\left( \pi \right) ^{i}x=0\text{ \ in norm }\forall x\in \mathcal{H}_{\pi }.
\end{equation*}
\end{corollary}

Let $G$ be a locally compact group. For an arbitrary $f\in L^{p}\left(
G\right) $ $\left( 1<p<\infty \right) ,$ we put $f^{\vee }\left( g\right)
:=f\left( g^{-1}\right) $ and $\widetilde{f}:=\overline{f\left(
g^{-1}\right) }$. Notice that for every $u,\upsilon \in L^{2}\left( G\right)
,$ the function $u\ast \widetilde{\upsilon }$ is in $C_{0}\left( G\right) $
and
\begin{equation*}
\langle \mu ,u\ast \widetilde{\upsilon }\rangle =\langle \mu \ast \overline{%
\upsilon },\overline{u}\rangle \text{, \ }\forall \mu \in M\left( G\right) .
\end{equation*}%
It follows that $\left\{ u\ast \widetilde{\upsilon }:u,\upsilon \in
L^{2}\left( G\right) \right\} $ is linearly dense in $C_{0}\left( G\right) .$
Notice also that if $f\in L^{p}\left( G\right) $ $\left( p\neq 2\right) $
and $h\in L^{q}\left( G\right) $ ($1/p+1/q=1$), then $h\ast f^{\vee }\in
C_{0}\left( G\right) $ and
\begin{equation*}
\langle \mu ,h\ast f^{\vee }\rangle =\langle \mu \ast f,h\rangle ,\text{ \ }%
\forall \mu \in M\left( G\right) .
\end{equation*}%
It follows that $\left\{ h\ast f^{\vee }:h\in L^{q}\left( G\right) ,\text{ }%
f\in L^{p}\left( G\right) \right\} $ is linearly dense in $C_{0}\left(
G\right) .$

Let $\pi $ be the left regular representation of $G$ on $L^{p}\left(
G\right) $ $\left( 1\leq p<\infty \right) $, where%
\begin{equation*}
\pi \left( g\right) f\left( s\right) =f\left( g^{-1}s\right) :=f_{g}\left(
s\right) .
\end{equation*}%
Then, $\pi $ is continuous and $\widehat{\mu }\left( \pi \right) f=\mu \ast
f $ for every $\mu \in M\left( G\right) .$ We will denote this operator by $%
\lambda _{p}\left( \mu \right) .$ It is well known that the left convolution
operator $\lambda _{p}\left( \mu \right) f:=\mu \ast f$ is a bounded linear
operator on $L^{p}\left( G\right) ,$ that is,%
\begin{equation*}
\left\Vert \lambda _{p}\left( \mu \right) f\right\Vert _{p}\leq \left\Vert
\mu \right\Vert _{1}\left\Vert f\right\Vert _{p}\text{ \ and \ }\left\Vert
\lambda _{1}\left( \mu \right) \right\Vert _{1}=\left\Vert \mu \right\Vert
_{1}.
\end{equation*}

From the identity (2.2) and Lemmas 2.3 and 2.4, we obtain the following.

\begin{corollary}
Let $\mu $ be a strictly aperiodic measure on a locally compact group $G$.
If $\left\{ a_{n}\right\} _{n\in
\mathbb{N}
}$ is a bounded good weight, then
\begin{equation*}
\lim_{n\rightarrow \infty }\frac{1}{n}\sum_{i=1}^{n}a_{i}\mu ^{i}\ast
f=\left( \lim_{n\rightarrow \infty }\frac{1}{n}\sum_{i=1}^{n}a_{i}\right)
P_{\mu }f\text{ \ in }L^{2}\text{-norm }\forall f\in L^{2}\left( G\right) ,
\end{equation*}%
where $P_{\mu }$ is the orthogonal projection onto
\begin{equation*}
\left\{ f\in L^{2}\left( G\right) :f_{g}=f:\forall g\in \left[ \func{supp}%
\mu \right] \right\} .
\end{equation*}%
If $\left[ \func{supp}\mu \right] $ is not compact, then
\begin{equation*}
\lim_{n\rightarrow \infty }\frac{1}{n}\sum_{i=1}^{n}a_{i}\mu ^{i}\ast f=0%
\text{ \ in }L^{2}\text{-norm for every }f\in L^{2}\left( G\right) .
\end{equation*}
\end{corollary}

It follows from Lemma 2.3 that if $\mu \in M\left( G\right) $ is a
probability measure and $1<p<\infty ,$ then for every $\xi \in \mathbb{T}$,%
\begin{equation*}
\ker \left[ \lambda _{p}\left( \mu \right) -\xi I\right] =\left\{ f\in
L^{p}\left( G\right) :f_{g}=\xi f,\text{ }\forall g\in \func{supp}\mu
\right\} .
\end{equation*}%
From this identity we can deduce that if the support of the probability
measure $\mu \in M\left( G\right) $ is not compact, then the operator $%
\lambda _{p}\left( \mu \right) $ $\left( 1<p<\infty \right) $ has no unitary
eigenvalues. On the other hand, we know \cite[Chapter II, Theorem 4.1]{7}
that if $T$ is a power bounded operator on a reflexive Banach space $X,$
then $T$ has no unitary eigenvalues if and only if
\begin{equation*}
\lim_{n\rightarrow \infty }\frac{1}{n}\dsum\limits_{i=1}^{n}\left\vert
\langle T^{n}x,\varphi \rangle \right\vert =0\text{ \ for all }x\in X\text{
and }\varphi \in X^{\ast }.
\end{equation*}%
\ \ Hence we have the following.

\begin{corollary}
If $\mu \in M\left( G\right) $ is a probability measure with non-compact
support, then for every $f\in L^{p}\left( G\right) $ $\left( 1<p<\infty
\right) $ and $h\in L^{q}\left( G\right) $ $\left( 1/p+1/q=1\right) ,$ we
have%
\begin{equation*}
\lim_{n\rightarrow \infty }\frac{1}{n}\dsum\limits_{i=1}^{n}\left\vert
\langle \mu ^{n}\ast f,h\rangle \right\vert =0.
\end{equation*}
\end{corollary}

We now in a position to prove Theorem 2.2.

\begin{proof}[Proof of theorem 2.2]
Let $H:=\left[ \func{supp}\mu \right] $. If
\begin{equation*}
m_{n}:=\frac{1}{n}\sum_{i=1}^{n}a_{i}\mu ^{i},
\end{equation*}%
then
\begin{equation*}
\sup_{n\in
\mathbb{N}
}\left\Vert m_{n}\right\Vert _{1}\leq \sup_{n\in
\mathbb{N}
}\left\{ \left\vert a_{n}\right\vert \right\} <\infty .
\end{equation*}%
(a) Let $\pi \in \widehat{G}$ and let $\left\{ e_{\pi }^{\left( 1\right)
},...,e_{\pi }^{\left( n_{\pi }\right) }\right\} $ be the basic vectors of
the representation space $\mathcal{H}_{\pi }$ ($\dim \mathcal{H}_{\pi
}=n_{\pi }$). Since
\begin{equation*}
\widehat{m_{n}}\left( \pi \right) =\frac{1}{n}\sum_{i=1}^{n}a_{i}\widehat{%
\mu }\left( \pi \right) ^{i},
\end{equation*}%
by Corollary 2.6,
\begin{equation*}
\lim_{n\rightarrow \infty }\widehat{m_{n}}\left( \pi \right) =aP_{\mu }^{\pi
}\text{ \ in the strong operator topology,}
\end{equation*}%
\ where%
\begin{equation*}
a:=\lim_{n\rightarrow \infty }\frac{1}{n}\sum_{i=1}^{n}a_{i}
\end{equation*}%
and $P_{\mu }^{\pi }$ is the orthogonal projection onto $\ker \left[
\widehat{\mu }\left( \pi \right) -I\right] .$ Let $f_{i,j}^{\pi }$ be the
matrix functions of $\pi $, where
\begin{equation*}
f_{i,j}^{\pi }\left( g\right) =\langle \pi \left( g\right) e_{\pi }^{\left(
i\right) },e_{\pi }^{\left( j\right) }\rangle \text{ \ }\left(
i,j=1,...,n_{\pi }\right) .
\end{equation*}%
Then we can write
\begin{equation*}
\lim_{n\rightarrow \infty }\langle m_{n},f_{i,j}^{\pi }\rangle
=\lim_{n\rightarrow \infty }\langle \widehat{m_{n}}\left( \pi \right) e_{\pi
}^{\left( i\right) },e_{\pi }^{\left( j\right) }\rangle =\langle aP_{\mu
}^{\pi }e_{\pi }^{\left( i\right) },e_{\pi }^{\left( j\right) }\rangle .
\end{equation*}%
By the Peter-Weyl $C$-Theorem \cite[Chapter 4]{17}, the system of matrix
functions
\begin{equation*}
\left\{ f_{i,j}^{\pi }:\pi \in \widehat{G},\text{ }i,j=1,...,n_{\pi }\right\}
\end{equation*}%
is linearly dense in $C\left( G\right) ,$ the space of all complex valued
continuous functions on $G.$ Since the sequence $\left\{ m_{n}\right\}
_{n\in
\mathbb{N}
}$ is norm bounded, the limit $\lim_{n\rightarrow \infty }\langle
m_{n},f\rangle $ exists for all $f\in C\left( G\right) .$ If $a=0,$ then
there is nothing to prove. Because in this case, $\lim_{n\rightarrow \infty
}\langle m_{n},f\rangle =0$ for all $f\in C\left( G\right) $ and therefore, w%
$^{\ast }$-$\lim m_{n}=0.$ Hence, we may assume that $a\neq 0.$

Further, since%
\begin{equation*}
f\mapsto \lim_{n\rightarrow \infty }\langle m_{n},f\rangle
\end{equation*}%
is a bounded linear functional on $C\left( G\right) $, there exists $\theta
_{\mu }\in M\left( G\right) $ such that%
\begin{equation*}
\lim_{n\rightarrow \infty }\langle m_{n},f\rangle =\langle \theta _{\mu
},f\rangle \text{, \ }\forall f\in C\left( G\right) .
\end{equation*}%
So we have
\begin{equation*}
\text{w}^{\ast }-\lim_{n\rightarrow \infty }m_{n}=\theta _{\mu }.
\end{equation*}%
It follows that
\begin{equation*}
\widehat{m_{n}}\left( \pi \right) \rightarrow \widehat{\theta _{\mu }}\left(
\pi \right) \text{ \ in the weak operator topology for all }\pi \in \widehat{%
G}.
\end{equation*}%
Consequently we have $\widehat{\theta _{\mu }}\left( \pi \right) =aP_{\mu
}^{\pi }$ for all $\pi \in \widehat{G}$. Now, we must show that $\theta
_{\mu }=a\overline{m}_{H}.$ To see this, let $\pi \in \widehat{G}$ be given.
Since $\overline{m}_{H}$ is an idempotent measure, $\widehat{\overline{m}_{H}%
}\left( \pi \right) $ is an orthogonal projection. Since $\func{supp}%
\overline{m}_{H}=H,$ by Lemma 2.3,%
\begin{eqnarray*}
\widehat{\overline{m}_{H}}\left( \pi \right) \mathcal{H}_{\pi } &=&\ker %
\left[ \widehat{\overline{m}_{H}}\left( \pi \right) -I_{\pi }\right] \\
&=&\left\{ x\in \mathcal{H}_{\pi }:\pi \left( g\right) x=x,\text{ }\forall
g\in H\right\} .
\end{eqnarray*}%
For the same reason,%
\begin{eqnarray*}
\frac{1}{a}\widehat{\theta _{\mu }}\left( \pi \right) \mathcal{H}_{\pi }
&=&P_{\mu }^{\pi }\mathcal{H}_{\pi }=\ker \left[ \widehat{\mu }\left( \pi
\right) -I_{\pi }\right] \\
&=&\left\{ x\in \mathcal{H}_{\pi }:\pi \left( g\right) x=x,\text{ }\forall
g\in H\right\} .
\end{eqnarray*}%
So we have $\widehat{\theta _{\mu }}\left( \pi \right) =a\widehat{\overline{m%
}_{H}}\left( \pi \right) $ for all $\pi \in \widehat{G}.$ It follows that $%
\theta _{{}}=a\overline{m}_{H}.$

(b) Assume that $H$ is not compact. As we have noted above, if $u,\upsilon
\in L^{2}\left( G\right) ,$ then $u\ast \widetilde{\upsilon }\in C_{0}\left(
G\right) $ and the set $\left\{ u\ast \widetilde{\upsilon }:u,\upsilon \in
L^{2}\left( G\right) \right\} $ is linearly dense in $C_{0}\left( G\right) $%
. Since the sequence $\left\{ m_{n}\right\} _{n\in
\mathbb{N}
}$ is uniformly bounded, it suffices to show that%
\begin{equation*}
\lim_{n\rightarrow \infty }\langle m_{n},u\ast \widetilde{\upsilon }\rangle
=0.
\end{equation*}%
By Corollary 2.6,

\begin{equation*}
\lim_{n\rightarrow \infty }\left( m_{n}\ast \overline{\upsilon }\right)
=\lim_{n\rightarrow \infty }\frac{1}{n}\sum_{i=1}^{n}a_{i}\mu ^{i}\ast
\overline{\upsilon }=0\ \ \text{in }L^{2}\text{-norm.}
\end{equation*}%
So we have
\begin{equation*}
\langle m_{n},u\ast \widetilde{\upsilon }\rangle =\lim_{n\rightarrow \infty
}\langle m_{n}\ast \overline{\upsilon },\overline{u}\rangle =0.
\end{equation*}
\end{proof}

Next, we have the following.

\begin{proposition}
Let $G$ be a compact group and let $\left\{ \mu _{n}\right\} _{n\in
\mathbb{N}
}$ be a bounded sequence in $M\left( G\right) $. The following conditions
are equivalent:

$\left( a\right) $ $w^{\ast }$-$\lim_{n\rightarrow \infty }\mu _{n}=\mu $
for some $\mu \in M\left( G\right) .$

$\left( b\right) $ $\lim_{n\rightarrow \infty }\left( \mu _{n}\ast f\right)
=\mu \ast f$ uniformly on $G$ for every $f\in C\left( G\right) .$
\end{proposition}

\begin{proof}
(a)$\Rightarrow $(b) Let $\pi \in \widehat{G}$ and let $\mathcal{H}_{\pi }$
be the representation space of $\pi .$ If%
\begin{equation*}
f_{x,y}^{\pi }\left( g\right) :=\langle \pi \left( g\right) x,y\rangle \text{
\ }(x,y\in \mathcal{H}_{\pi }),
\end{equation*}%
then as
\begin{equation*}
\langle \theta ,f_{x,y}^{\pi }\rangle =\langle \widehat{\theta }\left( \pi
\right) x,y\rangle \text{, \ }\forall \theta \in M\left( G\right) ,
\end{equation*}%
we have
\begin{equation*}
\langle \widehat{\mu _{n}}\left( \pi \right) x,y\rangle =\langle \mu
_{n},f_{x,y}^{\pi }\rangle \rightarrow \langle \mu ,f_{x,y}^{\pi }\rangle
=\langle \widehat{\mu }\left( \pi \right) x,y\rangle \text{.}
\end{equation*}%
Consequently, $\widehat{\mu _{n}}\left( \pi \right) \rightarrow \widehat{\mu
}\left( \pi \right) $\ in the strong operator topology. A few lines of
computation show that
\begin{equation}
\left( \theta \ast f_{x,y}^{\pi }\right) \left( g\right) =\langle \pi \left(
g\right) x,\widehat{\theta }\left( \pi \right) y\rangle \text{, \ }\forall
\theta \in M\left( G\right) .  \label{2.3}
\end{equation}%
Now let $f\in C\left( G\right) $ be given. Since the system of matrix
functions is linearly dense in $C\left( G\right) ,$ for an arbitrary $%
\varepsilon >0$ there exist complex numbers $\lambda _{1},...,\lambda _{k}$
and $\pi _{1},...,\pi _{k}\in \widehat{G}$ such that
\begin{equation*}
\left\vert f\left( g\right) -\lambda _{1}\langle \pi _{1}\left( g\right)
x_{1},y_{1}\rangle -...-\lambda _{k}\langle \pi _{k}\left( g\right)
x_{k},y_{k}\rangle \right\vert <\varepsilon \text{ \ }\left( \forall g\in
G\right) ,
\end{equation*}%
where $x_{i},y_{i}\in \mathcal{H}_{\pi _{i}}$ $(i=1,...,k).$ By (2.3),
\begin{equation*}
\left\vert \left( \mu _{n}\ast f\right) \left( g\right) -\lambda _{1}\langle
\pi _{1}\left( g\right) x_{1},\widehat{\mu _{n}}\left( \pi _{1}\right)
y_{1}\rangle -...-\lambda _{k}\langle \pi _{k}\left( g\right) x_{k},\widehat{%
\mu _{n}}\left( \pi _{k}\right) y_{k}\rangle \right\vert <\varepsilon C
\end{equation*}%
and
\begin{equation*}
\left\vert \left( \mu \ast f\right) \left( g\right) -\lambda _{1}\langle \pi
_{1}\left( g\right) x_{1},\widehat{\mu }\left( \pi _{1}\right) y_{1}\rangle
-...-\lambda _{k}\langle \pi _{k}\left( g\right) x_{k},\widehat{\mu }\left(
\pi _{k}\right) y_{k}\rangle \right\vert <\varepsilon C,
\end{equation*}%
where $C:=\sup_{n\in
\mathbb{N}
}\left\Vert \mu _{n}\right\Vert .$ It follows that
\begin{eqnarray*}
\sup_{g\in G}\left\vert \left( \mu _{n}\ast f\right) \left( g\right) -\left(
\mu \ast f\right) \left( g\right) \right\vert &\leq &\left\vert \lambda
_{1}\right\vert \left\Vert \widehat{\mu _{n}}\left( \pi _{1}\right) y_{1}-%
\widehat{\mu }\left( \pi _{1}\right) y_{1}\right\Vert \left\Vert
x_{1}\right\Vert +... \\
&&+\left\vert \lambda _{k}\right\vert \left\Vert \widehat{\mu _{n}}\left(
\pi _{k}\right) y_{k}-\widehat{\mu }\left( \pi _{k}\right) y_{k}\right\Vert
\left\Vert x_{k}\right\Vert +2\varepsilon C\text{.}
\end{eqnarray*}%
Since $\widehat{\mu _{n}}\left( \pi \right) x\rightarrow \widehat{\mu }%
\left( \pi \right) x$ in norm for all $\pi \in \widehat{G}$ and $x\in
\mathcal{H}_{\pi },$ from the preceding inequality we can deduce that $\mu
_{n}\ast f\rightarrow \mu \ast f$ uniformly on $G.$

(b)$\Rightarrow $(a) For an arbitrary $f\in C\left( G\right) ,$%
\begin{equation*}
\int_{G}fd\mu _{n}-\int_{G}fd\mu =\left( \mu _{n}\ast f\right) \left(
e\right) -\left( \mu \ast f\right) \left( e\right) \rightarrow 0.
\end{equation*}
\end{proof}

From Theorem 2.2 and Proposition 2.8 we have the following.

\begin{corollary}
Let $\mu \in M\left( G\right) $ be an adapted and strictly aperiodic measure
on a compact group $G.$ If $\left\{ a_{n}\right\} _{n\in
\mathbb{N}
}$ is a bounded good weight, then for every $f\in C\left( G\right) ,$

\begin{equation*}
\lim_{n\rightarrow \infty }\frac{1}{n}\sum_{i=1}^{n}a_{i}\mu ^{i}\ast
f=\left( \lim_{n\rightarrow \infty }\frac{1}{n}\sum_{i=1}^{n}a_{i}\right)
\int_{G}fdm_{G}\text{ \ uniformly on }G.
\end{equation*}
\end{corollary}

Recall that a measure $\mu \in M\left( G\right) $ is said to be \textit{%
power bounded if }%
\begin{equation*}
C_{\mu }:=\sup_{n\geq 0}\left\Vert \mu ^{n}\right\Vert _{1}<\infty .
\end{equation*}%
If $\mu \in M\left( G\right) $ is power bounded, then so is the operator $%
\lambda _{p}\left( \mu \right) $ $\left( 1<p<\infty \right) .$

The following result was proved in \cite[Theorem 3.4]{10}. The same result
for locally compact abelian groups was obtained earlier in \cite[Proposition
2.5]{19}. For completeness we give the proof of this result. Our proof is
different.

\begin{proposition}
If $\mu \in M\left( G\right) $\textit{\ is power bounded, then there exists
an idempotent measure }$\theta _{\mu }$\textit{\ in }$M\left( G\right) $
\textit{such that}%
\begin{equation*}
w^{\ast }-\lim_{n\rightarrow \infty }\frac{1}{n}\sum_{i=1}^{n}\mu
^{i}=\theta _{\mu }
\end{equation*}%
$(\theta _{\mu }$\textit{\ will be called limit measure associated with }$%
\mu ).$
\end{proposition}

\begin{proof}
Let $u,\upsilon \in L^{2}\left( G\right) $. Since the operator $\lambda
_{2}\left( \mu \right) $ is mean ergodic, there exists $w\in L^{2}\left(
G\right) $ such that%
\begin{equation*}
\lim_{n\rightarrow \infty }\frac{1}{n}\sum_{i=1}^{n}\mu ^{i}\ast \overline{%
\upsilon }=w\text{ \ in }L^{2}\text{-norm.}
\end{equation*}%
Therefore we have%
\begin{equation*}
\lim_{n\rightarrow \infty }\langle \frac{1}{n}\sum_{i=1}^{n}\mu ^{i},u\ast
\widetilde{\upsilon }\rangle =\lim_{n\rightarrow \infty }\langle \frac{1}{n}%
\sum_{i=1}^{n}\mu ^{i}\ast \overline{\upsilon },\overline{u}\rangle =\langle
w,\overline{u}\rangle .
\end{equation*}%
Since the sequence $\left\{ \frac{1}{n}\sum_{i=1}^{n}\mu ^{i}\right\} _{n\in
\mathbb{N}
}$ is bounded and the set $\left\{ u\ast \widetilde{\upsilon }:u,\upsilon
\in L^{2}\left( G\right) \right\} $ is linearly dense in $C_{0}\left(
G\right) ,$ the limit $\lim_{n\rightarrow \infty }\langle m_{n},f\rangle \ \
$exists for all $f\in C_{0}\left( G\right) .$ Consequently, there exists $%
\theta _{\mu }\in M\left( G\right) $ such that%
\begin{equation*}
\lim_{n\rightarrow \infty }\langle \frac{1}{n}\sum_{i=1}^{n}\mu
^{i},f\rangle =\langle \theta _{\mu },f\rangle \text{, \ }\forall f\in
C_{0}\left( G\right) .
\end{equation*}%
So we have
\begin{equation*}
\text{w}^{\ast }-\lim_{n\rightarrow \infty }\frac{1}{n}\sum_{i=1}^{n}\mu
^{i}=\theta _{\mu }.
\end{equation*}%
It is easy to check that $\theta _{\mu }$ is an idempotent measure.
\end{proof}

\begin{remark}
Let $G$ be a locally compact abelian group and let $\mu \in M\left( G\right)
$\textit{\ }be\textit{\ }power bounded\textit{. }If $\theta _{\mu }$\textit{%
\ }is the limit measure associated with\textit{\ }$\mu ,$ then $\widehat{%
\theta _{\mu }}=\mathbf{1}_{\text{int}\mathcal{F}_{\mu }},$ where $\mathcal{F%
}_{\mu }:=\left\{ \chi \in \widehat{G}:\widehat{\mu }\left( \chi \right)
=1\right\} $ and $\mathbf{1}_{\text{int}\mathcal{F}_{\mu }}$ is the
characteristic function of int$\mathcal{F}_{\mu }$. Consequently, int$%
\mathcal{F}_{\mu }$ is a clopen subset of $\widehat{G}$ \cite[Proposition 2.5%
]{19}. It is easy to check that if $\mu $ is a probability measure, then $%
\mathcal{F}_{\mu }$ is a closed subgroup of $\widehat{G},$ that is, $%
\mathcal{F}_{\mu }=\left[ \func{supp}\mu \right] ^{\bot }.$
\end{remark}

\begin{corollary}
If $\mu \in M\left( G\right) $\textit{\ is power bounded and }$1<p<\infty ,$
\textit{then }%
\begin{equation*}
\lim_{n\rightarrow \infty }\frac{1}{n}\sum_{i=1}^{n}\mu ^{i}\ast f=\theta
_{\mu }\ast f\text{ \ in }L^{p}\text{-norm for every }f\in L^{p}\left(
G\right) ,
\end{equation*}%
where $\theta _{\mu }$\textit{\ is the limit measure associated with }$\mu $
$(f\rightarrow \theta _{\mu }\ast f$ is the mean ergodic projection
associated with $\lambda _{p}\left( \mu \right) ).$
\end{corollary}

\begin{proof}
Let $f\in L^{p}\left( G\right) $ and $h\in L^{q}\left( G\right) $ $%
(1/p+1/q=1).$ Since
\begin{equation*}
\text{w}^{\ast }-\lim_{n\rightarrow \infty }\frac{1}{n}\sum_{i=1}^{n}\mu
^{i}=\theta _{\mu }\text{ \ and \ }h\ast f^{\vee }\in C_{0}\left( G\right) ,
\end{equation*}%
we can write
\begin{eqnarray*}
\lim_{n\rightarrow \infty }\langle \frac{1}{n}\sum_{i=1}^{n}\mu ^{i}\ast
f,h\rangle &=&\lim_{n\rightarrow \infty }\langle \frac{1}{n}%
\sum_{i=1}^{n}\mu ^{i},h\ast f^{\vee }\rangle \\
&=&\langle \theta _{\mu },h\ast f^{\vee }\rangle =\langle \theta _{\mu }\ast
f,h\rangle .
\end{eqnarray*}%
This shows that
\begin{equation*}
\lim_{n\rightarrow \infty }\frac{1}{n}\sum_{i=1}^{n}\mu ^{i}\ast f=\theta
_{\mu }\ast f\ \ \text{weakly.}
\end{equation*}%
Since the operator $\lambda _{p}\left( \mu \right) $ is mean ergodic, we have%
\begin{equation*}
\lim \frac{1}{n}\sum_{i=1}^{n}\mu ^{i}\ast f=\theta _{\mu }\ast f\text{ \ in
}L^{p}\text{-norm.}
\end{equation*}
\end{proof}

Let $\mu \in M\left( G\right) $ be power bounded and $\xi \in \mathbb{T}$.
Denote by $\theta _{\mu }^{\xi }$ the limit measure associated with\textit{\
}$\xi \mu .$ By Proposition 2.10, $\theta _{\mu }^{\xi }$ is an idempotent
measure and
\begin{equation*}
\text{w}^{\ast }-\lim_{n\rightarrow \infty }\frac{1}{n}\sum_{i=1}^{n}\xi
^{i}\mu ^{i}=\theta _{\mu }^{\xi }.
\end{equation*}

\begin{theorem}
Let $G$ be a second countable locally compact group and let $\mu \in M\left(
G\right) $ be a measure with $\left\Vert \mu \right\Vert _{1}\leq 1.$ The
following assertions hold:

$\left( a\right) $ $\sigma _{p}\left( \lambda _{2}\left( \mu \right) \right)
\cap \mathbb{T}$ is at most countable.

$\left( b\right) $ If $\left\{ a_{n}\right\} _{n\in
\mathbb{N}
}$ is a bounded \textit{good weight and }$\sigma _{p}\left( \lambda
_{2}\left( \mu \right) \right) \cap \mathbb{T=}\left\{ \xi _{1},\xi
_{2},...\right\} ,$ then
\begin{equation*}
\text{w}^{\ast }-\lim_{n\rightarrow \infty }\frac{1}{n}\sum_{i=1}^{n}a_{i}%
\mu ^{i}=\sum_{i=1}^{\infty }a\left( \xi _{i}\right) \theta _{\mu }^{\xi
_{i}},
\end{equation*}%
where $\theta _{\mu }^{\xi _{i}}$ is the limit measure associated with $\xi
_{i}\mu $ and%
\begin{equation*}
\lim_{n\rightarrow \infty }\frac{1}{n}\sum_{k=1}^{n}a_{k}\xi
_{i}^{k}=a\left( \xi _{i}\right) .
\end{equation*}
\end{theorem}

\begin{proof}
(a) Since $L^{2}\left( G\right) $ is separable and $\lambda _{2}\left( \mu
\right) $ is a contraction, by the Jamison theorem \cite{12}, $\sigma
_{p}\left( \lambda _{2}\left( \mu \right) \right) \cap \mathbb{T}$ is at
most countable.

(b) Let $f\in L^{2}\left( G\right) $ and $\xi \in \mathbb{T}$ be given. By
Corollary 2.12,
\begin{equation*}
\frac{1}{n}\sum_{i=1}^{n}\xi ^{i}\lambda _{2}\left( \mu \right)
^{i}f\rightarrow \theta _{\mu }^{\xi }\ast f\ \ \text{in }L^{2}\text{-norm.}
\end{equation*}%
Notice that $f\mapsto \theta _{\mu }^{\xi }\ast f$ is the orthogonal
projection onto $\ker \left[ \lambda _{2}\left( \mu \right) -\xi I\right] .$
Taking into account the identity (2.2), for an arbitrary $u,\upsilon \in
L^{2}\left( G\right) $ we can write%
\begin{eqnarray*}
\lim_{n\rightarrow \infty }\langle \frac{1}{n}\sum_{i=1}^{n}a_{i}\mu
^{i},u\ast \widetilde{\upsilon }\rangle &=&\lim_{n\rightarrow \infty
}\langle \frac{1}{n}\sum_{i=1}^{n}a_{i}\lambda _{2}\left( \mu \right) ^{i}%
\overline{u},\overline{v}\rangle \\
&=&\langle \sum_{i=1}^{\infty }a\left( \xi _{i}\right) \theta _{\mu }^{\xi
_{i}}\ast \overline{u},\overline{v}\rangle \\
&=&\langle \sum_{i=1}^{\infty }a\left( \xi _{i}\right) \theta _{\mu }^{\xi
_{i}},u\ast \widetilde{\upsilon }\rangle .
\end{eqnarray*}%
Since%
\begin{equation*}
\sup_{n\in
\mathbb{N}
}\left\Vert \frac{1}{n}\sum_{i=1}^{n}a_{i}\mu ^{i}\right\Vert _{1}\leq
\sup_{n\in
\mathbb{N}
}\left\vert a_{n}\right\vert <\infty
\end{equation*}%
and the set $\left\{ u\ast \widetilde{\upsilon }:u,\upsilon \in L^{2}\left(
G\right) \right\} $ is linearly dense in $C_{0}\left( G\right) ,$ we have
\begin{equation*}
\text{w}^{\ast }-\lim_{n\rightarrow \infty }\frac{1}{n}\sum_{i=1}^{n}a_{i}%
\mu ^{i}=\sum_{i=1}^{\infty }a\left( \xi _{i}\right) \theta _{\mu }^{\xi
_{i}}.
\end{equation*}
\end{proof}

\section{The sequence $\left\{ \protect\mu ^{n}\right\} _{n\in
\mathbb{N}
}$}

As we have noted above, $\left\Vert \lambda _{1}\left( \mu \right)
\right\Vert =\left\Vert \mu \right\Vert _{1}$ for all $\mu \in M\left(
G\right) .$ Moreover, we have $\sigma \left( \lambda _{1}\left( \mu \right)
\right) =\sigma _{M\left( G\right) }\left( \mu \right) $ for all $\mu \in
M\left( G\right) ,$ where $\sigma _{M\left( G\right) }\left( \mu \right) $
is the spectrum of $\mu $ with respect to the algebra $M\left( G\right) .$

\begin{proposition}
If $\mu \in M\left( G\right) $ is power bounded with $\sigma \left( \lambda
_{1}\left( \mu \right) \right) \cap \mathbb{T\subseteq }\left\{ 1\right\} ,$
then there exists an idempotent measure $\theta \in M\left( G\right) $ such
that%
\begin{equation*}
w^{\ast }-\lim_{n\rightarrow \infty }\mu ^{n}=\theta .
\end{equation*}
\end{proposition}

\begin{proof}
It suffices to show that the sequence $\left\{ \mu ^{n}\right\} _{n\in
\mathbb{N}
}$ has only one weak$^{\ast }$ cluster point. Since $\sigma \left( \lambda
_{1}\left( \mu \right) \right) \cap \mathbb{T\subseteq }\left\{ 1\right\} $,
by the Katznelson-Tzafriri theorem mentioned above,
\begin{equation*}
\lim_{n\rightarrow \infty }\left\Vert \mu ^{n+1}-\mu ^{n}\right\Vert
_{1}=\lim_{n\rightarrow \infty }\left\Vert \lambda _{1}\left( \mu \right)
^{n+1}-\lambda _{1}\left( \mu \right) ^{n}\right\Vert =0.
\end{equation*}%
Assume that
\begin{equation*}
\theta _{1}=\text{w}^{\ast }\text{-}\lim_{\alpha }\mu ^{n_{\alpha }}\text{ \
and \ }\theta _{2}=\text{w}^{\ast }\text{-}\lim_{\beta }\mu ^{m_{\beta }},
\end{equation*}%
for two subnets $\left\{ \mu ^{n_{\alpha }}\right\} _{\alpha }$ and $\left\{
\mu ^{m_{\beta }}\right\} _{\beta }$ of $\left\{ \mu ^{n}\right\} _{n\in
\mathbb{N}
}.$ Since the multiplication on $M\left( G\right) $ is separately w$^{\ast }$%
-continuous,
\begin{equation*}
\mu \ast \theta _{1}=\theta _{1}\ast \mu =\text{w}^{\ast }\text{-}%
\lim_{\alpha }\mu ^{n_{\alpha }+1}.
\end{equation*}%
Consequently we have%
\begin{equation*}
\left\Vert \mu \ast \theta _{1}-\theta _{1}\right\Vert _{1}\leq \underline{%
\lim }_{\alpha }\left\Vert \mu ^{n_{\alpha }+1}-\mu ^{n_{\alpha
}}\right\Vert =0
\end{equation*}%
and%
\begin{equation*}
\left\Vert \theta _{1}\ast \mu -\theta _{1}\right\Vert _{1}\leq \underline{%
\lim }_{\alpha }\left\Vert \mu ^{n_{\alpha }+1}-\mu ^{n_{\alpha
}}\right\Vert =0.
\end{equation*}%
Hence $\mu \ast \theta _{1}=\theta _{1}\ast \mu =\theta _{1}.$ Now, passing
to the limit (in the w$^{\ast }$-topology) in the identities%
\begin{equation*}
\mu ^{m_{\beta }}\ast \theta _{1}=\theta _{1}\ast \mu ^{m_{\beta }}=\theta
_{1},
\end{equation*}%
we have $\theta _{2}\ast \theta _{1}=\theta _{1}\ast \theta _{2}=\theta
_{1}. $ Similarly, $\theta _{2}\ast \theta _{1}=\theta _{1}\ast \theta
_{2}=\theta _{2}.$ If $\theta :=\theta _{1}=\theta _{2},$ then $\theta
^{2}=\theta .$
\end{proof}

Let $G$ be a locally compact abelian group. It is well known that $M\left(
G\right) $ is a commutative semisimple unital Banach algebra, but $M\left(
G\right) $ fails to be (Shilov) regular, in general. However, there exists a
largest closed regular subalgebra of $M\left( G\right) $ which we will
denote by $M_{\text{reg}}\left( G\right) .$ Since the algebra $L^{1}\left(
G\right) $ and the discrete measure algebra $M_{d}\left( G\right) $ are
subalgebras of $M_{\text{reg}}\left( G\right) ,$ we have $L^{1}\left(
G\right) +M_{d}\left( G\right) \subseteq M_{\text{reg}}\left( G\right) ,$
but in general, $L^{1}\left( G\right) +M_{d}\left( G\right) \neq M_{\text{reg%
}}\left( G\right) $ \cite[Example 4.3.11]{16}. This shows that the algebra $%
M_{\text{reg}}\left( G\right) $ is remarkable large. Recall \cite[%
Proposition 4.12.5]{16} that if $\mu \in M_{\text{reg}}\left( G\right) $,
then the operator $\lambda _{1}\left( \mu \right) $ has natural spectrum,
that is, $\sigma \left( \lambda _{1}\left( \mu \right) \right) =\overline{%
\widehat{\mu }\left( \widehat{G}\right) }$. It follows that if $\mu \in M_{%
\text{reg}}\left( G\right) $ is a probability measure, then $\sigma \left(
\lambda _{1}\left( \mu \right) \right) \cap \mathbb{T\subseteq }\left\{
1\right\} $ if and only if for any neighborhood $V$ of $1,$ $\sup_{\chi
\notin V}\left\vert \widehat{\mu }\left( \chi \right) \right\vert <1.$

\begin{corollary}
If $\mu \in M\left( G\right) $ is power bounded, then there exists an
idempotent measure $\theta \in M\left( G\right) $ such that
\begin{equation*}
w^{\ast }-\lim_{n\rightarrow \infty }\left( \frac{\delta _{e}+\mu }{2}%
\right) ^{n}=\theta .
\end{equation*}%
Moreover, for an arbitrary $f\in L^{p}\left( G\right) $ $\left( 1<p<\infty
\right) $ we have

\begin{equation*}
\lim_{n\rightarrow \infty }\left[ \left( \frac{\delta _{e}+\mu }{2}\right)
^{n}\ast f\right] =\theta \ast f\text{ \ in }L^{p}\text{-norm.}
\end{equation*}
\end{corollary}

\begin{proof}
If $\nu :=\frac{\delta _{e}+\mu }{2}$, then $\nu $ is power bounded, that
is, $\sup_{n\in
\mathbb{N}
}\left\Vert \nu ^{n}\right\Vert _{1}\leq C_{\mu }$. Therefore, the operator $%
\lambda _{1}\left( \nu \right) =\frac{I+\lambda _{1}\left( \mu \right) }{2}$
is power bounded. Notice also that if $f\left( z\right) :=\frac{1+z}{2}$ $%
\left( z\in
\mathbb{C}
\right) ,$ then $f\left( 1\right) =1$ and $\left\vert f\left( z\right)
\right\vert <1$ for all $z\in \overline{\mathbb{D}}\diagdown \left\{
1\right\} .$ It follows from the spectral mapping theorem that $\sigma
\left( \lambda _{1}\left( \nu \right) \right) \cap \mathbb{T\subseteq }%
\left\{ 1\right\} .$ By Proposition 3.1, there exists an idempotent measure $%
\theta \in M\left( G\right) $ such that w$^{\ast }$-$\lim_{n\rightarrow
\infty }\nu ^{n}=\theta .$

As in the proof of Proposition 2.12 we have%
\begin{equation*}
\lim_{n\rightarrow \infty }\left( \nu ^{n}\ast f\right) =\theta \ast f\text{
\ weakly for every }f\in L^{p}\left( G\right) .
\end{equation*}%
On the other hand, by the Katznelson-Tzafriri theorem,
\begin{equation*}
\lim_{n\rightarrow \infty }\left\Vert \lambda _{1}\left( \nu \right)
^{n+1}-\lambda _{1}\left( \nu \right) ^{n}\right\Vert =0.
\end{equation*}%
Now, it follows from Proposition 2.1 that%
\begin{equation*}
\lim_{n\rightarrow \infty }\left( \nu ^{n}\ast f\right) =\theta \ast f\text{%
\ \ in }L^{p}\text{-norm.}
\end{equation*}
\end{proof}

As we have noted above, if $\mu $ is a strictly aperiodic measure on a
non-compact locally compact group $G,$ then w$^{\ast }$-$\lim_{n\rightarrow
\infty }\mu ^{n}=0$ \cite[Th\'{e}or\`{e}me 8]{3}. In \cite[Theorem 2]{18} it
was proved that if $\mu $ is an adapted measure on a second countable
locally compact non-compact group $G$, then w$^{\ast }$-$\lim_{n\rightarrow
\infty }\mu ^{n}=0$.

As is well known, equipped with the involution$\ $given by $d\widetilde{\mu }%
\left( g\right) =\overline{d\mu \left( g^{-1}\right) }$, the algebra $%
M\left( G\right) $ becomes a Banach $\ast $-algebra. If $\mu $ is a
probability measure on a locally compact group $G,$ then as $\func{supp}%
\widetilde{\mu }=\left( \func{supp}\mu \right) ^{-1}$, we have
\begin{equation*}
\func{supp}\left( \widetilde{\mu }\ast \mu \right) =\overline{\left\{ \left(
\func{supp}\mu \right) ^{-1}\cdot \left( \func{supp}\mu \right) \right\} }
\end{equation*}%
(for instance see, \cite[Theorem 2.2.2]{11}).

The following result is certainly known, but we have not been able to find a
precise reference.

\begin{proposition}
A measure $\mu \in M\left( G\right) $ is strictly aperiodic if and only if
the measure $\widetilde{\mu }\ast \mu $ is adapted.
\end{proposition}

\begin{proof}
Assume that the measure $\widetilde{\mu }\ast \mu $ is not adapted. This
means that
\begin{equation*}
\left[ \func{supp}\left( \widetilde{\mu }\ast \mu \right) \right] :=H\neq G.
\end{equation*}%
It follows that $g^{-1}\cdot \left( \func{supp}\mu \right) \subseteq H$ or $%
\func{supp}\mu \subseteq gH$ for all $g\in \func{supp}\mu .$ This shows that
$\mu $ is not strictly aperiodic. Now, assume that $\mu $ is not strictly
aperiodic. Then $\func{supp}\mu \subseteq gH$ for some closed proper
subgroup $H$ of $G$ and $g\in G$. It follows that $\left( \func{supp}\mu
\right) ^{-1}\subseteq Hg^{-1}$ and therefore $\left( \func{supp}\mu \right)
^{-1}\cdot \left( \func{supp}\mu \right) \subseteq H.$ Thus we have
\begin{equation*}
\left[ \func{supp}\left( \widetilde{\mu }\ast \mu \right) \right] \subseteq
H\neq G.
\end{equation*}%
This shows that the measure $\widetilde{\mu }\ast \mu $ is not adapted.
\end{proof}

Next, we have the following.

\begin{theorem}
Let $\mu $ be a probability measure on a locally compact group $G.$ If one
of the subgroups $\left[ \func{supp}\left( \widetilde{\mu }\ast \mu \right) %
\right] $ and $\left[ \func{supp}\left( \mu \ast \widetilde{\mu }\right) %
\right] $ is not compact, then
\begin{equation*}
w^{\ast }-\lim_{n\rightarrow \infty }\mu ^{n}=0.
\end{equation*}
\end{theorem}

\begin{proof}
Recall that a contraction $T$ on a Hilbert space is said to be \textit{%
completely non-unitary} if it has no proper reducing subspace on which it
acts as a unitary operator. By the Nagy-Foia\c{s} theorem \cite[Chapter II,
Theorem 3.9]{7} if $T$ is a completely non-unitary contraction, then $%
T^{n}\rightarrow 0$ in the weak operator topology. Now, assume that $\left[
\func{supp}\left( \widetilde{\mu }\ast \mu \right) \right] $ is not compact.
Let us show that $\lambda _{2}\left( \mu \right) $ is a completely
non-unitary contraction. Since $\lambda _{2}\left( \mu \right) ^{\ast
}=\lambda _{2}\left( \widetilde{\mu }\right) ,$ it suffices to show that $%
\lambda _{2}\left( \widetilde{\mu }\right) \lambda _{2}\left( \mu \right)
f=f $, where $f\in L^{2}\left( G\right) $ implies $f=0.$ By Lemma 2.3,%
\begin{eqnarray*}
\left\{ f\in L^{2}\left( G\right) :\lambda _{2}\left( \widetilde{\mu }%
\right) \lambda _{2}\left( \mu \right) f=f\right\} &=&\left\{ f\in
L^{2}\left( G\right) :\lambda _{2}\left( \widetilde{\mu }\ast \mu \right)
f=f\right\} \\
&=&\left\{ f\in L^{2}\left( G\right) :f_{g}=f\text{, }\forall g\in \left[
\func{supp}\left( \widetilde{\mu }\ast \mu \right) \right] \right\} .
\end{eqnarray*}%
Since $\left[ \func{supp}\left( \widetilde{\mu }\ast \mu \right) \right] $
is not compact, from the identity $f_{g}=f$ for all $g\in \left[ \func{supp}%
\left( \widetilde{\mu }\ast \mu \right) \right] $, we have $f=0$ (a.e.).
Hence $\lambda _{2}\left( \mu \right) $ is a completely non-unitary
contraction. Consequently, $\lambda _{2}\left( \mu \right) ^{n}\rightarrow 0$
in the weak operator topology. Then, for an arbitrary $u,\upsilon \in
L^{2}\left( G\right) $ we can write
\begin{equation*}
\langle \mu ^{n},u\ast \widetilde{\upsilon }\rangle =\langle \mu ^{n}\ast
\overline{\upsilon },\overline{u}\rangle =\langle \lambda _{2}\left( \mu
\right) ^{n}\overline{\upsilon },\overline{u}\rangle \rightarrow 0.
\end{equation*}%
Since the set $\left\{ u\ast \widetilde{\upsilon }:u,\upsilon \in
L^{2}\left( G\right) \right\} $ is linearly dense in $C_{0}\left( G\right) ,$
we have $\langle \mu ^{n},f\rangle \rightarrow 0$ for all $f\in C_{0}\left(
G\right) .$ Hence w$^{\ast }$-$\lim_{n\rightarrow \infty }\mu ^{n}=0.$
\end{proof}

\begin{corollary}
Let $\mu $ be a probability measure on a locally compact group $G.$ If one
of the subgroups $\left[ \func{supp}\left( \widetilde{\mu }\ast \mu \right) %
\right] $ and $\left[ \func{supp}\left( \mu \ast \widetilde{\mu }\right) %
\right] $ is not compact, then
\begin{equation*}
\mu ^{n}\ast f\rightarrow 0\text{ \ weakly for all }f\in L^{p}\left(
G\right) \text{ }\left( 1<p<\infty \right) .
\end{equation*}
\end{corollary}

\begin{proof}
If $f\in L^{p}\left( G\right) $ and $h\in L^{q}\left( G\right) $ $%
(1/p+1/q=1),$ then $h\ast f^{\vee }\in C_{0}\left( G\right) $. By Theorem
3.4, w$^{\ast }$-$\lim_{n\rightarrow \infty }\mu ^{n}=0$ and therefore we
have
\begin{equation*}
\lim_{n\rightarrow \infty }\langle \mu ^{n}\ast f,h\rangle
=\lim_{n\rightarrow \infty }\langle \mu ^{n},h\ast f^{\vee }\rangle =0.
\end{equation*}
\end{proof}

However, we have the following.

\begin{proposition}
Let $\mu $ be a probability measure on a locally compact group $G$ and
assume that one of the subgroups $\left[ \func{supp}\left( \widetilde{\mu }%
\ast \mu \right) \right] $ and $\left[ \func{supp}\left( \mu \ast \widetilde{%
\mu }\right) \right] $ is not compact. If the Lebesgue measure of $\sigma
\left( \lambda _{2}\left( \mu \right) \right) \cap \mathbb{T}$ is zero, then
\begin{equation*}
\lim_{n\rightarrow \infty }\left\Vert \mu ^{n}\ast f\right\Vert _{p}=0,\text{
\ }\forall f\in L^{p}\left( G\right) \text{ }\left( 1<p<\infty \right) .
\end{equation*}
\end{proposition}

\begin{proof}
As in the proof of Theorem 3.4 we can see that if one of the subgroups $%
\left[ \func{supp}\left( \widetilde{\mu }\ast \mu \right) \right] $ and $%
\left[ \func{supp}\left( \mu \ast \widetilde{\mu }\right) \right] $ is not
compact, then $\lambda _{2}\left( \mu \right) $ is a completely non-unitary
contraction. On the other hand, Nagy-Foias theorem \cite[Chapter II,
Proposition 6.7]{7} states that if $T$ is a completely non-unitary
contraction on a Hilbert space $\mathcal{H}$ and if $\sigma \left( T\right)
\cap \mathbb{T}$ is of the Lebesgue measure zero, then $\lim_{n\rightarrow
\infty }\left\Vert T^{n}x\right\Vert =0$ for all $x\in \mathcal{H}.$
Therefore we have $\lim_{n\rightarrow \infty }\left\Vert \mu ^{n}\ast
f\right\Vert _{2}=0$ for all $f\in L^{2}\left( G\right) .$ Now, assume that $%
p\neq 2.$ Let $C_{c}\left( G\right) $ be the space of all complex valued
continuous functions on $G$ with compact support and let $f\in C_{c}\left(
G\right) .$ Then, $\nu \ast f\in L^{p}\left( G\right) $ for any $\nu \in
M\left( G\right) $ and for all $1\leq p\leq \infty .$ By the Riesz-Thorin
Convexity Theorem \cite[Chapter VI, Lemma 10.9]{6}, $\alpha \mapsto \log
\left\Vert \nu \ast f\right\Vert _{\frac{1}{\alpha }}$ is a convex function
on $\left[ 0,1\right] $. Let $q$ be chosen such that $q>p$ for $p>2$ and $%
1<q<p$ for $1<p<2.$ If $\lambda :=\frac{2q-2p}{pq-2p},$ then $0<\lambda <1$
and $\frac{1}{p}=\frac{1-\lambda }{q}+\frac{\lambda }{2}.$ Consequently, we
can write
\begin{equation*}
\left\Vert \nu \ast f\right\Vert _{p}\leq \left\Vert \nu \ast f\right\Vert
_{q}^{1-\lambda }\left\Vert \nu \ast f\right\Vert _{2}^{\lambda },\text{ \ }%
\forall \nu \in M\left( G\right) .
\end{equation*}%
Hence we have
\begin{equation*}
\left\Vert \mu ^{n}\ast f\right\Vert _{p}\leq \left\Vert \mu ^{n}\ast
f\right\Vert _{q}^{1-\lambda }\left\Vert \mu ^{n}\ast f\right\Vert
_{2}^{\lambda }\leq \left\Vert f\right\Vert _{q}^{1-\lambda }\left\Vert \mu
^{n}\ast f\right\Vert _{2}^{\lambda },\text{ \ }\forall n\in
\mathbb{N}
.
\end{equation*}%
It follows that $\lim_{n\rightarrow \infty }\left\Vert \mu ^{n}\ast
f\right\Vert _{p}=0$ for all $f\in C_{c}\left( G\right) .$ Since $%
C_{c}\left( G\right) $ is dense in $L^{p}\left( G\right) ,$ we have $%
\lim_{n\rightarrow \infty }\left\Vert \mu ^{n}\ast f\right\Vert _{p}=0$ for
all $f\in L^{p}\left( G\right) .$
\end{proof}

\begin{remark}
Let $G$ be a locally compact abelian group and $\mu \in M\left( G\right) .$
The Fourier-Plancherel transform estabilishes a unitary equivalence between
the convolution operator $\lambda _{2}\left( \mu \right) $ on $L^{2}\left(
G\right) $ and the multiplication operator $M_{\widehat{\mu }}$ on $%
L^{2}\left( \widehat{G}\right) $. Therefore, we have $\sigma \left( \lambda
_{2}\left( \mu \right) \right) =\overline{\widehat{\mu }\left( \widehat{G}%
\right) }.$ If $\mathcal{E}_{\mu }:=\left\{ \chi \in \widehat{G}:\left\vert
\widehat{\mu }\left( \chi \right) \right\vert =1\right\} $, then $\mathcal{E}%
_{\mu }=\mathcal{F}_{\widetilde{\mu }\ast \mu }$. Notice also that  $%
\widehat{\mu }\left( \mathcal{E}_{\mu }\right) \subseteq \sigma \left(
\lambda _{2}\left( \mu \right) \right) \cap \mathbb{T}$. If $\mu $ is a
probability measure, then $\mathcal{E}_{\mu }$ is a closed subgroup of $%
\widehat{G},$ that is, $\mathcal{E}_{\mu }=\left[ \func{supp}\left(
\widetilde{\mu }\ast \mu \right) \right] ^{\bot }$ $($see, Remark 2.11$).$
\end{remark}

\end{document}